\documentclass{amsart}
\usepackage{graphicx, color}
\usepackage{amscd}
\usepackage{amsmath,empheq}
\usepackage{amsfonts}
\usepackage{amssymb}
\usepackage{mathrsfs}
\usepackage[all]{xy}
\newtheorem{theorem}{Theorem}

\newtheorem{prop}[theorem]{Proposition}
\newtheorem{remark}{Remark}

\newtheorem{claim}{Claim}

\newenvironment{proof-sketch}{\noindent{\bf Sketch of Proof}\hspace*{1em}}{\qed\bigskip}
\everymath{\displaystyle}
\newcommand{\RR}{\mathbb R}
\newcommand{\NN}{\mathbb N}

\renewcommand{\leq}{\leqslant}

\renewcommand{\geq}{\geqslant}
\begin{document}
\title[Nonlinear elliptic inclusions]{Nonlinear elliptic inclusions with unilateral constraint and dependence on the gradient}
\author[N.S. Papageorgiou]{Nikolaos S. Papageorgiou}
\address[N.S. Papageorgiou]{National Technical University, Department of Mathematics,
				Zografou Campus, 15780 Athens, Greece}
\email{\tt npapg@math.ntua.gr}
\author[V.D. R\u{a}dulescu]{Vicen\c{t}iu D. R\u{a}dulescu}
\address[V.D. R\u{a}dulescu]{Department of Mathematics, Faculty of Sciences, King Abdulaziz University, P.O. Box 80203, Jeddah 21589, Saudi Arabia \&  Department of Mathematics, University of Craiova, Street A.I. Cuza 13,
          200585 Craiova, Romania}
\email{\tt vicentiu.radulescu@imar.ro}
\author[D.D. Repov\v{s}]{Du\v{s}an D. Repov\v{s}}
\address[D.D. Repov\v{s}]{Faculty of Education and Faculty of Mathematics and Physics,
University of Ljubljana, 1000 Ljubljana, Slovenia}\email{dusan.repovs@guest.arnes.si}
\begin{abstract}
We consider a nonlinear Neumann elliptic inclusion with a source (reaction term) consisting of a convex subdifferential plus a multivalued term depending on the gradient. The convex subdifferential incorporates in our framework problems with unilateral constraints (variational inequalities). Using topological methods and the Moreau-Yosida approximations of the subdifferential term, we establish the existence of a smooth solution.
\end{abstract}
\maketitle
\section{Introduction}

Let $\Omega\subseteq\RR^N$ be a bounded domain with a $C^2$-boundary $\partial\Omega$. In this paper we study the following nonlinear Neumann elliptic differential inclusion
\begin{equation}\label{eq1}
	\left\{\begin{array}{ll}
		{\rm div}\,(a(u(z))Du(z))\in\partial\varphi(u(z))+F(z,u(z),Du(z))&\mbox{in}\ \Omega,\\
		\frac{\partial u}{\partial n}=0&\mbox{on}\ \partial\Omega.
	\end{array}\right\}
\end{equation}

In this problem, $\varphi\in\Gamma_0(\RR)$ (that is, $\varphi:\RR\rightarrow\overline{\RR}=\RR\cup\{+\infty\}$ is proper, convex and lower semicontinuous, see Section 2) and $\partial\varphi(x)$ is the subdifferential of $\varphi(\cdot)$ in the sense of convex analysis. Also $F(z,x,\xi)$ is a multivalued term with closed convex values depending on the gradient of $u$. So, problem (\ref{eq1}) incorporates variational inequalities with a multivalued reaction term.

By a solution of problem (\ref{eq1}), we understand a function $u\in H^1(\Omega)$ such that we can find $g,f\in L^2(\Omega)$ for which we have
\begin{eqnarray*}
	&&g(z)\in\partial\varphi(u(z))\ \mbox{and}\ f(z)\in F(z,u(z),Du(z))\ \mbox{for almost all}\ z\in\Omega,\\
	&&\int_{\Omega}a(u(z))(Du,Dh)_{\RR^N}dz+\int_{\Omega}(g(z)+f(z))h(z)dz=0\ \mbox{for all}\ h\in H^1(\Omega).
\end{eqnarray*}

The presence of the gradient in the multifunction $F$, precludes the use of variational methods in the analysis of (\ref{eq1}). To deal with such problems, a variety of methods have been proposed. Indicatively, we mention the works of Amann and Crandall \cite{1}, de Figueiredo, Girardi and Matzeu \cite{6}, Girardi and Matzeu \cite{9}, Loc and Schmitt \cite{14}, Pohozaev \cite{21}. All these papers consider problems with no unilateral constraint (that is, $\varphi=0$) and the reaction term $F$ is single-valued. Variational inequalities (that is, problems where $\varphi$ is the indicator function of a closed, convex set), were investigated by Arcoya, Carmona and Martinez Aparicio \cite{2}, Matzeu and Servadei \cite{16}, Mokrane and Murat \cite{18}. All have single valued source term.

Our method of proof is topological and it is based on a slight variant of Theorem 8 of Bader \cite{3} (a multivalued alternative theorem). Also, our method uses approximations of $\varphi$ and the theory of nonlinear operators of monotone type. In the next section, we recall the basic notions and mathematical tools which we will use in the sequel.

\section{Mathematical Background}

Let $X$ be a Banach space and $X^*$ be its topological dual. By $\left\langle \cdot,\cdot\right\rangle$ we denote the duality brackets for the pair $(X^*,X)$. By $\Gamma_0(X)$ we denote the cone of all convex functions $\varphi:X\rightarrow\overline{\RR}=\RR\cup\{+\infty\}$ which are proper (that is, not identically $+\infty$) and lower semicontinuous. By ${\rm dom}\,\varphi$ we denote the effective domain of $\varphi$, that is, $${\rm dom}\,\varphi :=\{u\in X:\varphi(u)<+\infty\}.$$ Given $\varphi\in\Gamma_0(X)$, the subdifferential of $\varphi$ at $u\in X$ is the set
$$\partial\varphi(u)=\{u^*\in X^*:\left\langle u^*,h\right\rangle\leq\varphi(u+h)-\varphi(u)\ \mbox{for all}\ h\in X\}.$$

Evidently $\partial\varphi(u)\subseteq X^*$ is $w^*$-closed, convex and possibly empty. If $\varphi$ is continuous at $u\in X$, then $\partial\varphi(u)\subseteq X^*$ is nonempty, $w^*$-compact and convex. Moreover, if $\varphi$ is G\^ateaux differentiable at $u\in X$, then $\partial\varphi(u)=\{\varphi'_G(u)\}$ ($\varphi'_G(u)$ being the G\^ateaux derivative of $\varphi$ at $u$). We know that the map $\partial \varphi:X\rightarrow 2^{X^*}$ is maximal monotone. If $X=H=$ a Hilbert space and $\varphi\in\Gamma_0(H)$, then for every $\lambda>0$, the ``Moreau-Yosida approximation'' $\varphi_{\lambda}$ of $\varphi$, is defined by
$$\varphi_{\lambda}(u)=\inf\left[\varphi(h)+\frac{1}{2\lambda}||h-u||^2:h\in H\right]\ \mbox{for all}\ u\in H.$$

We have the following properties:
\begin{itemize}
	\item $\varphi_{\lambda}$ is convex, ${\rm dom}\,\varphi_{\lambda}=H$;
	\item $\varphi_{\lambda}$ is Fr\'echet differentiable and the Fr\'echet derivative $\varphi'_{\lambda}$ is Lipschitz continuous with Lipschitz constant $1/\lambda$;
	\item if $\lambda_n\rightarrow 0,\ u_n\rightarrow u$ in $H$, $\varphi'_{\lambda_n}(u_n)\stackrel{w^*}{\rightarrow}u^*$ in $H$, then $u^*\in\partial\varphi(u)$.
\end{itemize}

We refer for details to Gasinski and Papageorgiou \cite{7} and Papageorgiou and Kyritsi \cite{20}.

We know that if $\varphi\in\Gamma_0(X)$, then $\varphi$ is locally Lipschitz in the interior of its effective domain (that is, on ${\rm int}\,{\rm dom}\,\varphi$). So, locally Lipschitz functions are the natural candidate to extend the subdifferential theory of convex functions.

We say that $\varphi:X\rightarrow\RR$ is locally Lipschitz if for every $u\in X$ we can find $U$ a neighborhood of $u$ and a constant $k>0$ such that
$$|\varphi(v)-\varphi(y)|\leq k||v-y||\ \mbox{for all}\ v,y\in U.$$

For such functions we can define the generalized directional derivative $\varphi^0(u;h)$ by
$$\varphi^{\circ}(u;h)=\limsup\limits_{\stackrel{u'\rightarrow u}{\lambda\downarrow 0}}\frac{\varphi(u'+\lambda h)-\varphi(u')}{\lambda}\,.$$

Then $\varphi^{\circ}(u;\cdot)$ is sublinear continuous and so we can define the nonempty $w^*$-compact set $\partial_c\varphi(u)$ by
$$\partial_c\varphi(u)=\{u^*\in X^*:\left\langle u^*,h\right\rangle\leq\varphi^{\circ}(u;h)\ \mbox{for all}\ h\in X\}.$$

We say that $\partial_c\varphi(u)$ is the ``Clarke subdifferential'' of $\varphi$ at $u\in X$. In contrast to the convex subdifferential, the Clarke subdifferential is always nonempty. Moreover, if $\varphi$ is convex, continuous (hence locally Lipschitz on $X$), then the two subdifferentials coincide, that is, $\partial\varphi(u)=\partial_c\varphi(u)$ for all $u\in X$. For further details we refer to Clarke \cite{5}.

Suppose that $X$ is a reflexive Banach space and $A:X\rightarrow X^*$ a map. We say that $A$ is ``pseudomonotone'', if the following two conditions hold:
\begin{itemize}
	\item $A$ is continuous from every finite dimensional subspace $V$ of $X$ into $X^*$ furnished with the weak topology;
	\item if $u_n\stackrel{w}{\rightarrow}u$ in $X$, $A(u_n)\stackrel{w}{\rightarrow}u^*$ in $X^*$ and $\limsup\limits_{n\rightarrow\infty}\left\langle A(u_n),u_n-u\right\rangle\leq 0$, then for every $y\in X$, we have
	$$\left\langle A(u),u-y\right\rangle\leq\liminf\limits_{n\rightarrow\infty}\left\langle A(u_n),u_n-y\right\rangle .$$
\end{itemize}

If $A:X\rightarrow X^*$ is maximal monotone, then $A$ is pseudomonotone.

A pseudomonotone map $A:X\rightarrow X^*$ which is strongly coercive, that is,
$$\frac{\left\langle A(u),u\right\rangle}{||u||}\rightarrow+\infty\ \mbox{as}\ ||u||\rightarrow\infty,$$
it is surjective (see Gasinski and Papageorgiou \cite[p. 336]{7}).

Let $V$ be a set and let $G:V\rightarrow 2^{X^*}\backslash\{\emptyset\}$ be a multifunction. The graph of $G$ is the set
$${\rm Gr}\, G=\{(v,u)\in V\times X:u\in G(v)\}.$$

$(a)$ If $V$ is a Hausdorff topological space and ${\rm Gr}\,G\subseteq V\times X$ is closed, then we say that $G$ is ``closed''.

$(b)$ If there is a $\sigma$-field $\Sigma$ defined on $V$ and ${\rm Gr}\,G\subseteq\Sigma\times B(X)$, with $B(X)$ being the Borel $\sigma$-field of $X$, then we say that $G$ is ``graph measurable".

As we already mentioned in the Introduction, our approach uses a slight variant of Theorem 8 of Bader \cite{3} in which the Banach space $V$ is replaced by its dual $V^*$ equipped with the $w^*$-topology. A careful reading of the proof of Bader \cite{3}, reveals that the result remains true if we make this change.

So, as above $X$ is a Banach space, $V^*$ is a dual Banach space, $G:X\rightarrow 2^{V^*}$ is a multifunction with nonempty, $w^*$-compact, convex values. We assume that $G(\cdot)$ is ``upper semicontinuous'' (usc for short), from $X$ with the norm topology into $V^*$ with the $w^*$-topology (denoted by $V^*_{w^*}$), that is, for all $U\subseteq V^*$ $w^*$-open, we have
$$G^-(U)=\{x\in X:G(x)\cap U\neq\emptyset\}\ \mbox{is open}.$$

Note that if ${\rm Gr}\,G\subseteq X\times V^*_{w^*}$ is closed and $G(\cdot)$ is locally compact into $V^*_{w^*}$, that is, for all $u\in X$ we can find $U$ a neighborhood of $u$ such that $\overline{G(U)}^{w^*}$ is $w^*$-compact in $V^*$, then $G$ is usc from $X$ into $V^*_{w^*}$. Also, let $K:V^*_{w^*}\rightarrow X$ be a sequentially continuous map. Then the nonlinear alternative theorem of Bader \cite{3}, reads as follows.
\begin{theorem}\label{th1}
	Assume that $G$ and $K$ are as above and $S=K\circ G:X\rightarrow 2^X\backslash\{\emptyset\}$  maps bounded sets into relatively compact sets. Define
	$$E=\{u\in X:u\in tS(u)\ \mbox{for some}\ t\in(0,1)\}.$$
	Then either $E$ is unbounded or $S(\cdot)$ admits a fixed point.
\end{theorem}

\section{Existence Theorem}

In this section we prove an existence theorem for problem (\ref{eq1}). We start by introducing the hypotheses on the data of problem (\ref{eq1}).

\smallskip
$H(a)$: $a:\RR\rightarrow\RR$ is a function which satisfies
\begin{eqnarray*}
	&&|a(x)-a(y)|\leq k|x-y|\ \mbox{for all}\ x,y\in\RR,\ \mbox{some}\ k>0,\\
	&&0<c_1\leq a(x)\leq c_2\ \mbox{for all}\ x\in\RR.
\end{eqnarray*}

$H(\varphi)$: $\varphi\in\Gamma_0(\RR)$ and $0\in\partial\varphi(0)$.
\begin{remark}
	We recall that in $\RR\times\RR$, every maximal monotone set is of the subdifferential type. In higher dimensions this is no longer true (see Papageorgiou and Kyritsi \cite[p. 175]{20}).
\end{remark}

$H(F)$: $F:\Omega\times\RR\times\RR^N\rightarrow P_{f_c}(\RR)$ is a multifunction such that
\begin{itemize}
	\item[(i)] for all $(x,\xi)\in \RR\times\RR^N,\ z\mapsto F(z,x,\xi)$ is graph measurable;
	\item[(ii)] for almost all $z\in\Omega,\ (x,\xi)\mapsto F(z,x,\xi)$ is closed;
	\item[(iii)] for almost all $z\in\Omega$ and all $(x,\xi,v)\in {\rm Gr}\,F(z,\cdot,\cdot)$, we have
	$$|v|\leq\gamma_1(z,|x|)+\gamma_2(z,|x|)|\xi|$$
	with \begin{eqnarray*}
		&&\sup[\gamma_1(z,s):0\leq s\leq k]\leq\eta_{1,k}(z)\ \mbox{for almost all}\ z\in\Omega,\\
		&&\sup[\gamma_2(z,s):0\leq s\leq k]\leq\eta_{2,k}(z)\ \mbox{for almost all}\ z\in\Omega,\\
		\mbox{and}&&\eta_{1,k},\eta_{2,k}\in L^{\infty}(\Omega);
	\end{eqnarray*}
	\item[(iv)] there exists $M>0$ such that if $|x_0|>M$, then we can find $\delta>0$ and $\eta>0$ such that
	$$\inf[vx+c_1|\xi|^2:|x-x_0|+|\xi|\leq\delta,v\in F(z,x,\xi)]\geq\eta>0\ \mbox{for almost all}\ z\in\Omega,$$
	with $c_1>0$ as in hypothesis $H(a)$;
	\item[(v)] for almost all $z\in\Omega$ and all $(x,\xi,v)\in {\rm Gr}\,F(z,\cdot,\cdot)$, we have
	$$vx\geq-c_3|x|^2-c_4|x|\xi|-\gamma_3(z)|x|$$
	with $c_3,c_4>0$ and $\gamma_3\in L^1(\Omega)_+$.
\end{itemize}
\begin{remark}
	Hypothesis $H(F)(iv)$ is an extension to multifunctions of the Nagumo-Hartman condition for continuous vector fields (see Hartman \cite[p. 433]{10}, Knobloch \cite{12} and Mawhin \cite{17}).
\end{remark}

Let $\hat{a}:H^1(\Omega)\rightarrow H^1(\Omega)^*$ be the nonlinear continuous map defined by
\begin{equation}\label{eq2}
	\left\langle \hat{a}(u),h\right\rangle=\int_{\Omega}a(u)(Du,Dh)_{\RR^N}dz\ \mbox{for all}\ u,h\in H^1(\Omega).
\end{equation}
\begin{prop}\label{prop2}
	If hypotheses $H(a)$ hold, then the map $\hat{a}:H^1(\Omega)\rightarrow H^1(\Omega)^*$ defined by (\ref{eq2}) is pseudomonotone.
\end{prop}
\begin{proof}
	Evidently $\hat{a}(\cdot)$ is bounded (that is, maps bounded sets to bounded sets), see hypotheses $H(a)$ and it is defined on all of $H^1(\Omega)$. So, in order to prove  the desired pseudomonotonicity of $\hat{a}(\cdot)$, it suffices to show the following:
	
	\textit{(GP):} ``If $u_n\stackrel{w}{\rightarrow}u$ in $H^1(\Omega),\ \hat{a}(u_n)\stackrel{w}{\rightarrow}u^*$ in $H^1(\Omega)^*$ and $\limsup\limits_{n\rightarrow\infty}\left\langle \hat{a}(u_n),u_n-u\right\rangle\leq 0$,\\ then $u^*=\hat{a}(u)$ and $\left\langle \hat{a}(u_n),u_n\right\rangle\rightarrow\left\langle \hat{a}(u),u\right\rangle$''\\ (see Gasinski and Papageorgiou \cite{7}, Proposition 3.2.49, p. 333).
	
	So, according to (GP) above we consider a sequence $\{u_n\}_{n\geq 1}\subseteq H^1(\Omega)$ such that
	\begin{equation}\label{eq3}
		u_n\stackrel{w}{\rightarrow}u\ \mbox{in}\ H^1(\Omega),\ \hat{a}(u_n)\stackrel{w}{\rightarrow}u^*\ \mbox{in}\ H^1(\Omega)^*\ \mbox{and}\ \limsup\limits_{n\rightarrow\infty}\left\langle \hat{a}(u_n),u_n-u\right\rangle\leq 0.
	\end{equation}
	
	We have
	\begin{eqnarray}\label{eq4}
		\left\langle \hat{a}(u_n),u_n-u\right\rangle&=&\int_{\Omega}a(u_n)(Du_n,Du_n-Du)_{\RR^N}dz\nonumber\\
		&=&\int_{\Omega}a(u_n)|Du_n-Du|^2dz+\int_{\Omega}a(u_n)(Du,Du_n-Du)_{\RR^N}dz.
	\end{eqnarray}
	
	Hypotheses $H(a)$ and (\ref{eq3}) imply that
	\begin{equation}\label{eq5}
		\int_{\Omega}a(u_n)(Du,Du_n-Du)_{\RR^N}dz\rightarrow 0\ \mbox{as}\ n\rightarrow\infty.
	\end{equation}
	
	Also we have
	\begin{eqnarray}\label{eq6}
		&&\int_{\Omega}a(u_n)|Du_n-Du|^2dz\geq c_1||Du_n-Du||^2_2\ (\mbox{see hypotheses}\ H(a)),\nonumber\\
		&\Rightarrow&Du_n\rightarrow Du\ \mbox{in}\ L^2(\Omega,\RR^N)\ (\mbox{see (\ref{eq3}), (\ref{eq4}), \eqref{eq5}})\nonumber\\
		&\Rightarrow&u_n\rightarrow u\ \mbox{in}\ H^1(\Omega)\ (\mbox{see (\ref{eq3})}).
	\end{eqnarray}
	
	For all $h\in H^1(\Omega)$, we have
	\begin{eqnarray*}
		&&\left\langle \hat{a}(u_n),h\right\rangle=\int_{\Omega}a(u_n)(Du_n,Dh)_{\RR^N}dz\rightarrow
\int_{\Omega}a(u)(Du,Dh)_{\RR^N}dz=\left\langle \hat{a}(u),h\right\rangle\\
		&&\hspace{6cm}(\mbox{see (\ref{eq3}) and hypotheses}\ H(a)),\\
		&\Rightarrow& \hat{a}(u_n)\stackrel{w}{\rightarrow}\hat{a}(u)\ \mbox{in}\ H^1(\Omega)^*,\\
		&\Rightarrow&\hat{a}(u)=u^*\ (\mbox{see (\ref{eq3})}).
	\end{eqnarray*}
	
	From (\ref{eq6}) and the continuity of $a(\cdot)$ (see hypotheses $H(a)$), we have
	$$\left\langle \hat{a}(u_n),u_n\right\rangle\rightarrow\left\langle \hat{a}(u),u\right\rangle\,.$$
	
	Therefore property (GP) is satisfied and so we conclude that $\hat{a}(\cdot)$ is pseudomonotone.
\end{proof}

Next we will approximate problem (\ref{eq1}) using the Moreau-Yosida approximations of $\varphi\in\Gamma_0(\RR)$. For this approach to lead to a solution of problem (\ref{eq1}), we need to have a priori bounds for the approximate solutions. The proposition which follows is a crucial step in this direction. Its proof is based on the so-called ``Morse iteration technique''.

So, we consider the following nonlinear Neumann problem:
\begin{eqnarray}\label{eq7}
	\left\{\begin{array}{ll}
		-{\rm div}\,(a(u(z))Du(z))=g(z,u(z))&\mbox{in}\ \Omega,\\
		\frac{\partial u}{\partial n}=0&\mbox{on}\ \partial\Omega.
	\end{array}\right\}
\end{eqnarray}

The conditions on the reaction term $g(z,x)$ are the following:

\smallskip
$H(g):$ $g:\Omega\times\RR\rightarrow\RR$ is a Carath\'eodory function (that is, for all $x\in\RR$, $z\mapsto g(z,x)$ is measurable and for almost all $z\in\Omega$, $x\mapsto g(z,x)$ is continuous) and
$$|g(z,x)|\leq a(z)(1+|x|^{r-1})\ \mbox{for almost all}\ z\in\Omega,\ \mbox{all}\ x\in\RR,$$
with $a\in L^{\infty}(\Omega)_+$, $2\leq r<2^*=\left\{\begin{array}{ll}
	\frac{2N}{N-2}&\mbox{if}\ N\geq 3\\
	+\infty&\mbox{if}\ N=1,2
\end{array}\right.$ (the critical Sobolev exponent).

By a weak solution of problem (\ref{eq7}), we understand a function $u\in H^1(\Omega)$ such that
$$\int_{\Omega}a(u)(Du,Dh)_{\RR^N}dz=\int_{\Omega}g(z,u)hdz\ \mbox{for all}\ h\in H^1(\Omega).$$
\begin{prop}\label{prop3}
	If hypothesis $H(g)$ holds and $u\in H^1(\Omega)$ is a nontrivial weak solution of (\ref{eq7}), then $u\in L^{\infty}(\Omega)$ and $||u||_{\infty}\leq M=M(||a||_{\infty},N,2,||u||_{2^*})$.
\end{prop}
\begin{proof}
	Let $p_0=2^*$ and $p_{n+1}=2^*+\frac{2^*}{2}(p_n-r)$ for all $n\in\NN_0$. Evidently $\{p_n\}_{n\geq 0}$ is increasing. First suppose that $u\geq 0$. For every $k\in\NN$ we set
	\begin{equation}\label{eq8}
		u_k=\min\{u,k\}\in H^1(\Omega)\,.
	\end{equation}
	
	Let $\vartheta=p_n-r>0$ (note that $p_n\geq 2^*>r$). We have
	\begin{equation}\label{eq9}
		\hat{a}(u)=N_g(u)\ \mbox{in}\ H^1(\Omega)^*
	\end{equation}
	with $N_g(u)(\cdot)=g(\cdot,u(\cdot))\in L^{r'}(\Omega)\subseteq H^1(\Omega)^*,\frac{1}{r}+\frac{1}{r'}=1$ (the Nemytskii map corresponding to $g$). On (\ref{eq9}) we act with $u^{\vartheta+1}_k$ (see (\ref{eq8})). Then
	\begin{equation}\label{eq10}
		\left\langle \hat{a}(u),u^{\vartheta+1}_k\right\rangle=\int_{\Omega}g(z,u)u^{\vartheta+1}_kdz.
	\end{equation}
	
	Note that
	\begin{eqnarray}\label{eq11}
		&&\left\langle \hat{a}(u),u^{\vartheta+1}_k\right\rangle=\int_{\Omega}a(u)(Du,Du^{\vartheta+1}_k)_{\RR^N}dz\nonumber\\
		&=&(\vartheta+1)\int_{\Omega}u^{\vartheta}_ka(u)(Du,Du_k)_{\RR^N}dz\nonumber\\
		&\geq&(\vartheta+1)\int_{\Omega}u^{\vartheta}_kc_1|Du_k|^2dz\ (\mbox{see hypothesis H(a) and recall that}\ u\geq 0)\nonumber\\
		&=&c_1(\vartheta+1)\frac{2}{\vartheta+2}\int_{\Omega}|Du_k^{\frac{\vartheta+2}{2}}|^2dz.
	\end{eqnarray}
	
	Also we have
	\begin{eqnarray}\label{eq12}
		&&\int_{\Omega}g(z,u)u^{\vartheta+1}_kdz\nonumber\\
		&\leq&\int_{\Omega}a(z)(1+u^{r-1})u^{\vartheta+1}dz\ (\mbox{see hypothesis $H(g)$, (\ref{eq8}) and recall}\ u\geq 0)\nonumber\\
		&\leq&c_3(1+\int_{\Omega}u^{p_n}dz)\ \mbox{for some}\ c_3>0\ (\mbox{since}\ \vartheta+1<\vartheta+r=p_n).
	\end{eqnarray}
	
	We return to (\ref{eq10}) and use (\ref{eq11}) and (\ref{eq12}). Then
	\begin{eqnarray*}
		&&c_1(\vartheta+1)\frac{2}{\vartheta+2}\int_{\Omega}\left[|Du_k^{\frac{\vartheta+2}{2}}|^2+|u_k^{\frac{\vartheta+2}{2}}|^2\right]dz\\
		&\leq&c_4(1+\int_{\Omega}u^{p_n}dz)\ \mbox{for some}\ c_4>0\ (\mbox{since}\ \vartheta+r=p_n)\\
		\Rightarrow&&||u_k^{\frac{\vartheta+2}{2}}||^2\leq c_5(1+\int_{\Omega}u^{p_n}dz)\ \mbox{for some}\ c_5>0,\ \mbox{all}\ k\in\NN,\ \mbox{and}\ n\in\NN_0.
	\end{eqnarray*}
	
	Here $||\cdot||$ denotes the norm of $H^1(\Omega)$ (recall that $||v||=[||v||^2_2+||Dv||^2_2]^{1/2}$ for all $v\in H^1(\Omega)$).
	
	By the Sobolev embedding theorem (see (\ref{eq8}) and note that $H^1(\Omega)\hookrightarrow L^{\frac{2_{p_{n+1}}}{p_n}}(\Omega)$) we have
	$$||u_k||^{p_n}_{p_{n+1}}\leq c_6(1+\int_{\Omega}u^{p_n}dz)\ \mbox{for some}\ c_6>0,\ \mbox{all}\ k\in\NN_0\ \mbox{and}\ n\in\NN .$$
	
	Let $k\rightarrow\infty$. Then $u_k(z)\uparrow u(z)$ for almost all $z\in\Omega$ (see (\ref{eq8})). So, by the monotone convergence theorem, we have
	\begin{equation}\label{eq13}
		\left(\int_{\Omega}u^{p_{n+1}}dz\right)^{\frac{p_n}{p_{n+1}}}\leq c_6\left(1+\int_{\Omega}u^{p_n}dz\right)\ \mbox{for all}\ n\in\NN_0.
	\end{equation}
	
	Recall that $p_0=2^*$ and by the Sobolev embedding theorem we have $u\in L^{2^*}(\Omega)$. So, from (\ref{eq13}) and by induction we infer that $u\in L^{p_n}(\Omega)$ for all $n\in\NN_0$. Also we have
	$$||u||^{p_n}_{p_{n+1}}\leq c_6(1+||u||^{p_n}_{p_n})\ \mbox{for all}\ n\in\NN_0\ (\mbox{see (\ref{eq13})}).$$
	
	Since $p_n<p_{n+1}$, using the H\"{o}lder and Young inequalities (the latter with $\epsilon>0$ small), we obtain
	\begin{equation}\label{eq14}
		||u||_{p_n}\leq c_7\ \mbox{for some}\ c_7>0,\ \mbox{all}\ n\in \NN_0.
	\end{equation}
	\begin{claim}
		$p_n\rightarrow\infty$.
	\end{claim}
	
	Arguing by contradiction, suppose that the claim were not true. Since $\{p_n\}_{n\in\NN_0}$ is increasing, we have
	\begin{equation}\label{eq15}
		p_n\rightarrow p_*>2^*.
	\end{equation}
	
	By definition
	\begin{eqnarray*}
		&&p_{n+1}=2^*+\frac{2^*}{2}(p_n-r),\\
		&\Rightarrow&p_*=2^*+\frac{2^*}{2}(p_*-r)\ (\mbox{see (\ref{eq15})})\\
		&\Rightarrow&p_*\left(\frac{2^*}{2}-1\right)=2^*\left(\frac{r}{2}-1\right)<2^*\left(\frac{2^*}{2}-1\right)\ (\mbox{since}\ 2\leq r<2^*),\\
		&\Rightarrow&p_*<2^*,\ \mbox{a contradiction (see \ref{eq15})}.
	\end{eqnarray*}
	
	This proves Claim 1.
	
	So, passing to the limit as $n\rightarrow\infty$ in (\ref{eq14}), it follows from Gasinski and Papageorgiou \cite[p. 477]{8} that
	$$||u||_{\infty}\leq c_7,\ \mbox{hence}\ u\in L^{\infty}(\Omega).$$
	
	Moreover, it is clear from the above proof that $||u||_{\infty}\leq M=M(||a||_{\infty},N,2,||u||_{2^*})$.
	
	Finally for the general case, we write $u=u^+-u^-$, with $u^{\pm}=\max\{\pm u,0\}\geq 0$ and work with each one separately as above, to conclude $u^{\pm}\in L^{\infty}(\Omega)$, hence $u\in L^{\infty}(\Omega)$.
\end{proof}

Now for $\lambda>0$, let $\varphi_{\lambda}$ be the Moreau-Yosida approximation of $\varphi\in\Gamma_0(\RR)$ and for $\vartheta\in L^{\infty}(\Omega)$, consider the following auxiliary Neumann problem:
\begin{equation}\label{eq16}
	\left\{\begin{array}{ll}
		-{\rm div}\,(a(u(z))Du(z))+u(z)+\varphi'_{\lambda}(u(z))=\vartheta(z)&\mbox{in}\ \Omega,\\
		\frac{\partial u}{\partial n}=0&\mbox{on}\ \partial\Omega
	\end{array}\right\}
\end{equation}
\begin{prop}\label{prop4}
	If hypotheses $H(a),H(\varphi)$ hold and $\vartheta\in L^{\infty}(\Omega)$, then problem (\ref{eq16}) admits a unique solution $u\in C^1(\overline{\Omega})$.
\end{prop}
\begin{proof}
	Let $V_{\lambda}:H^1(\Omega)\rightarrow H^1(\Omega)^*$ be the nonlinear map defined by
	$$V_{\lambda}(u)=\hat{a}(u)+u+N_{\varphi'_{\lambda}}(u)\ \mbox{for all}\ u\in H^1(\Omega).$$
	
	As before $N_{\varphi'_{\lambda}}(u)$ is the Nemytskii map corresponding to $\varphi'_{\lambda}$ (that is, $N_{\varphi'_{\lambda}}(u)(\cdot)=\varphi'_{\lambda}(u(\cdot))$). We have
	\begin{eqnarray}\label{eq17}
		\left\langle V_{\lambda}(u),u\right\rangle&=&\left\langle \hat{a}(u),u\right\rangle+||u||^2_2+\int_{\Omega}\varphi'_{\lambda}(u)udz\nonumber\\
		&\geq&c_1||Du||^2_2+||u||^2_2\nonumber\\
		&&(\mbox{see hypothesis $H(a)$ and recall that}\ \varphi'_{\lambda}\ \mbox{is increasing,}\ \varphi'_{\lambda}(0)=0),\nonumber\\
		\Rightarrow&&V_{\lambda}\ \mbox{is strongly coercive}.
	\end{eqnarray}
	
	Using the Sobolev embedding theorem we see that $u\mapsto N_{\varphi'_{\lambda}}(u)$ is completely continuous from $H^1(\Omega)$ into $H^1(\Omega)^*$ (that is, if $u_n\stackrel{w}{\rightarrow}u$ in $H^1(\Omega)$, then $N_{\varphi'_{\lambda}}(u_n)\rightarrow N_{\varphi'_{\lambda}}(u)$ in $H^1(\Omega)^*$), hence it is pseudomonotone. From Proposition \ref{prop2} we know that $\hat{a}(\cdot)$ is pseudomonotone and of course the same is true for the embedding $H^1(\Omega)\hookrightarrow H^1(\Omega)^*$ (which is compact). So, from Gasinski and Papageorgiou \cite{7}, Proposition 3.2.51, p. 334, we infer that
	\begin{equation}\label{eq18}
		u\mapsto V_{\lambda}(u)\ \mbox{is pseudomonotone}.
	\end{equation}
	
	Recall that a pseudomonotone strongly coercive map is surjective. So, from (\ref{eq17}), (\ref{eq18}) it follows that there exists $u\in H^1(\Omega)$ such that
	\begin{eqnarray}\label{eq19}
		&&V_{\lambda}(u)=\vartheta,\nonumber\\
		&\Rightarrow&\int_{\Omega}a(u)(Du,Dh)_{\RR^N}dz+\int_{\Omega}uhdz+\int_{\Omega}\varphi'_{\lambda}(u)hdz=\int_{\Omega}\vartheta hdz\ \mbox{for all}\ h\in H^1(\Omega).
	\end{eqnarray}
	
	From the nonlinear Green's identity (see Gasinski and Papageorgiou \cite{7}, Theorem 2.4.53, p. 210), we have
	\begin{equation}\label{eq20}
		\int_{\Omega}a(u)(Du,Dh)_{\RR^N}dz=\left\langle -{\rm div}\,(a(u)Du),h\right\rangle+\left\langle a(u)\frac{\partial u}{\partial n},h\right\rangle_{\partial\Omega}\ \mbox{for all}\ h\in H^1(\Omega),
	\end{equation}
	where by $\left\langle \cdot,\cdot\right\rangle_{\partial\Omega}$ we denote the duality brackets for the pair $(H^{-\frac{1}{2},2}(\partial\Omega),H^{\frac{1}{2},2}(\partial\Omega))$.
	
	From the representation theorem for the elements of $H^{-1}(\Omega)=H^1_0(\Omega)^*$ (see Gasinski and Papageorgiou \cite{7}, Theorem 2.4.57, p. 212), we have
	$${\rm div}\,(a(u)Du)\in H^{-1}(\Omega).$$
	
	So, if by $\left\langle \cdot,\cdot\right\rangle_0$ we denote the duality brackets for the pair $(H^{-1}(\Omega),H^1_0(\Omega))$ we have
	\begin{eqnarray}\label{eq21}
		&&\left\langle -{\rm div}\,(a(u)Du),h\right\rangle_0=\int_{\Omega}a(u)(Du,Dh)_{\RR^N}dz\ \mbox{for all}\ h\in H^1_0(\Omega),\nonumber\\
		&\Rightarrow&\left\langle -{\rm div}\,(a(u)Du),h\right\rangle_0=\int_{\Omega}(\vartheta-u-\varphi'_{\lambda}(u))hdz\ \mbox{for all}\ h\in H^1_0(\Omega)\ (\mbox{see (\ref{eq19})}),\nonumber\\
		&\Rightarrow&-{\rm div}\,(a(u(z))Du(z))=\vartheta(z)-u(z)-\varphi'_{\lambda}(u(z))\ \mbox{for almost all}\ z\in\Omega .
	\end{eqnarray}
	
	Then from (\ref{eq19}), (\ref{eq20}), (\ref{eq21}) it follows that
	\begin{equation}\label{eq22}
		\left\langle a(u)\frac{\partial u}{\partial n},h\right\rangle_{\partial\Omega}=0\ \mbox{for all}\ h\in H^1(\Omega).
	\end{equation}
	
	If by $\gamma_0$ we denote the trace map, we recall that
	$${\rm im}\,\gamma_0=H^{\frac{1}{2},2}(\partial\Omega)$$
	(see Gasinski and Papageorgiou \cite{7}, Theorem 2.4.50, p. 209). Hence from (\ref{eq22}) we infer that
	$$\left.\frac{\partial u}{\partial n}\right|_{\partial\Omega}=0\ (\mbox{see hypothesis $H(a)$}).$$
	
	Therefore we have
	\begin{equation}\label{eq23}
		\left\{\begin{array}{ll}
			-{\rm div}\,(a(u)(z)Du(z))+u(z)+\varphi'_{\lambda}(u(z))=\vartheta(z)&\mbox{for almost all}\ z\in\Omega,\\
			\frac{\partial u}{\partial n}=0&\mbox{on}\ \partial\Omega
		\end{array}\right\}
	\end{equation}
	
	From (\ref{eq23}) and Proposition \ref{prop3}, we infer that
	$$u\in L^{\infty}(\Omega).$$
	
	Then we can use Theorem 2 of Lieberman \cite{13} and conclude that
	$$u\in C^1(\overline{\Omega}).$$
	
 We establish in what follows the uniqueness of this solution. So, suppose that $v\in C^1(\overline{\Omega})$ is another solution. We have
	\begin{eqnarray}
		&&\hat{a}(u)+u+N_{\varphi'_{\lambda}}(u)=\vartheta\ \mbox{in}\ H^1(\Omega)^*,\label{eq24}\\
		&&\hat{a}(v)+v+N_{\varphi'_{\lambda}}(v)=\vartheta\ \mbox{in}\ H^1(\Omega)^*.\label{eq25}
	\end{eqnarray}
	
	Let $k>0$ be the Lipschitz constant in hypothesis $H(a)$. We introduce the following function
	\begin{eqnarray}\label{eq26}
		\eta_{\epsilon}(s)=\left\{\begin{array}{ll}
			\int^s_{\epsilon}\frac{dt}{(kt)^2}&\mbox{is}\ s\geq\epsilon\\
			0&\mbox{if}\ s<\epsilon
		\end{array}\right.\ \mbox{with}\ \epsilon>0.
	\end{eqnarray}
	
	Evidently $\eta_{\epsilon}$ is Lipschitz continuous. So, from Marcus and Mizel \cite{15}, we have
	\begin{eqnarray}
		&&\eta_{\epsilon}(u-v)\in H^1(\Omega),\label{eq27}\\
		&&D(\eta_{\epsilon}(u-v))=\eta'_{\epsilon}(u-v)D(u-v)\label{eq28}
	\end{eqnarray}
	(see also Gasinski and Papageorgiou \cite{7}, Proposition 2.4.25, p. 195). Subtracting (\ref{eq25}) from (\ref{eq24}), we have
	\begin{eqnarray}\label{eq29}
		\hat{a}(u)-\hat{a}(v)+(u-v)+(N_{\varphi'_{\lambda}}(u)-N_{\varphi'_{\lambda}(v)})=0\ \mbox{in}\ H^1(\Omega)^*.
	\end{eqnarray}
	
	On (\ref{eq29}) we act with $\eta_{\epsilon}(u-v)\in H^1(\Omega)$ (see (\ref{eq27})). Then
	\begin{equation}\label{eq30}
		\left\langle \hat{a}(u)-\hat{a}(v),\eta_{\epsilon}(u-v)\right\rangle+\int_{\Omega}(u-v)\eta_{\epsilon}(u-v)dz+
\int_{\Omega}(\varphi'_{\lambda}(u)-\varphi'_{\lambda}(v))(u-v)dz=0.
	\end{equation}
	
	We have
	\begin{equation}\label{eq31}
		\int_{\Omega}(u-v)\eta_{\epsilon}(u-v)dz=\int_{\{u-v\geq\epsilon\}}(u-v)\eta_{\epsilon}(u-v)dz\geq
\frac{1}{k}\int_{\{u-v\geq\epsilon\}}\left(\frac{u-v}{\epsilon}-1\right)dz\ (\mbox{see (\ref{eq26})}).
	\end{equation}
	
	Recall that $\varphi'_{\lambda}$ is increasing. Therefore
	\begin{eqnarray}\label{eq32}
		&&\int_{\Omega}(\varphi'_{\lambda}(u)-\varphi'_{\lambda}(v))\eta_{\epsilon}(u-v)dz=\int_{\{u-v\geq\epsilon\}}(\varphi'_{\lambda}(u)-\varphi'_{\lambda}(v))\eta_{\epsilon}(u-v)dz\geq 0\\
		&&\mbox{(see (\ref{eq26}))}.\nonumber
	\end{eqnarray}
	
	We return to (\ref{eq30}) and use (\ref{eq31}), (\ref{eq32}). Then
	\begin{eqnarray}\label{eq33}
		&&\left\langle \hat{a}(u)-\hat{a}(v),\eta_{\epsilon}(u-v)\right\rangle\leq 0,\nonumber\\
		&\Rightarrow&\int_{\Omega}(a(u)Du-a(v)Dv,D\eta_{\epsilon}(u-v))_{\RR^N}dz\leq 0,\nonumber\\
		&\Rightarrow&\int_{\Omega}a(u)(Du-Dv,D\eta_{\epsilon}(u-v))_{\RR^N}dz\leq-
\int_{\Omega}(a(u)-a(v))(Dv,D\eta_{\epsilon}(u-v))_{\RR^N}dz.
	\end{eqnarray}
	
	Let $\Omega_{\epsilon}=\{z\in\Omega:(u-v)(z)\geq\epsilon\}$. Then
	\begin{eqnarray}\label{eq34}
		&&\int_{\Omega}a(u)(Du-Dv,D\eta_{\epsilon}(u-v))_{\RR^N}dz\nonumber\\
		&&=\int_{\Omega_{\epsilon}}a(u)\eta'_{\epsilon}(u-v)|Du-Dv|^2dz\ (\mbox{see (\ref{eq26}), (\ref{eq28})})\nonumber\\
		&&\geq c_1\int_{\Omega_{\epsilon}}\frac{|Du-Dv|^2}{k^2(u-v)^2}dz\ (\mbox{see hypothesis $H(a)$ and (\ref{eq26})}).
	\end{eqnarray}
	
	Also we have
	\begin{eqnarray}\label{eq35}
		&&-\int_{\Omega}(a(u)-a(v))(Dv,D\eta_{\epsilon}(u-v))_{\RR^N}dz\nonumber\\
		&&\leq\int_{\Omega_{\epsilon}}k(u-v)\eta'_{\epsilon}(u-v)(Dv,Du-Dv)_{\RR^N}dz\ (\mbox{see hypothesis $H(a)$ and (\ref{eq28})})\nonumber\\
		&&=\int_{\Omega_{\epsilon}}\frac{1}{k(u-v)}(Dv,Du-Dv)_{\RR^N}dz\ (\mbox{see (\ref{eq26})})\nonumber\\
		&&\leq||Dv||_2\left(\int_{\Omega_{\epsilon}}\frac{|Du-Dv|^2}{k^2|u-v|^2}dz\right)^{1/2}\ (\mbox{by the Cauchy-Schwarz inequality}).
	\end{eqnarray}
	
	Returning to (\ref{eq33}) and using (\ref{eq34}), (\ref{eq35}) we obtain
	$$\int_{\Omega_{\epsilon}}\frac{|Du-Dv|^2}{|u-v|^2}dz\leq\frac{k^2}{c_1^2}||Dv||^2_2.$$
	
	Let $\Omega^*_{\epsilon}$ be a connected component of $\hat\Omega=\{z\in\Omega;\ (u-v)(z)>0\}$, $\hat\Omega\not=\Omega$ (see \eqref{eq31}). We have
	\begin{equation}\label{eq36}
		\int_{\Omega^*_{\epsilon}}\frac{|Du-Dv|^2}{|u-v|^2}dz\leq\frac{k^2}{c_1^2}||Dv||^2_2\quad\mbox{with $\Omega^*_{\epsilon}=\Omega_\epsilon\cap\Omega^*$}\,.
	\end{equation}
	
	Consider the function
	\begin{eqnarray}\label{eq37}
		\gamma_{\epsilon}(y)=\left\{\begin{array}{ll}
			\int^y_{\epsilon}\frac{dt}{t}&\mbox{if}\ t\geq\epsilon\\
			0&\mbox{if}\ t<\epsilon .
		\end{array}\right.
	\end{eqnarray}
	
	This function is Lipschitz continuous and as before from Marcus and Mizel \cite{15}, we have
	\begin{eqnarray}
		&&\gamma_{\epsilon}(u-v)\in H^1(\Omega)\label{eq38}\\
		&&D\gamma_{\epsilon}(u-v)=\gamma'_{\epsilon}(u-v)(Du-Dv)=\frac{1}{u-v}(Du-Dv)\ \mbox{for almost all}\ z\in\Omega_{\epsilon}\ (\mbox{see (\ref{eq37})})\label{eq39}.
	\end{eqnarray}
	
	Returning to (\ref{eq36}) and using (\ref{eq38}), (\ref{eq39}), we obtain
	\begin{equation}\label{eq40}
		\int_{\Omega^*}|D\gamma_{\epsilon}(u-v)|^2dz\leq\frac{k^2}{c_1^2}||Dv||^2_2\,.
	\end{equation}
	
	Note that $u=v$ on $\partial\Omega^*$ (that is, $u-v\in H^1_0(\Omega^*)$; recall that $u,v\in C^1(\overline{\Omega})$). Hence
	\begin{equation}\label{eq41}
		\gamma_{\epsilon}(u-v)\in H^1_0(\Omega^*).
	\end{equation}
	
	From (\ref{eq40}), (\ref{eq41}) and the Poincar\'e inequality, we have
	$$\int_{\Omega^*}|\gamma_{\epsilon}(u-v)|^2dz\leq c_8||v||^2\ \mbox{for some}\ c_8>0,\ \mbox{all}\ \epsilon>0.$$
	
	If $|\Omega^*|_N>0$ (by $|\cdot|_N$ we denote the Lebesgue measure on $\RR^N$), then letting $\epsilon\rightarrow 0^+$, we reach a contradiction (see (\ref{eq37})). So, every connected component of the open set
	$$\hat{\Omega}=\{z\in\Omega:u(z)>v(z)\}$$
	is Lebesgue-null. Hence $|\hat{\Omega}|_N=0$ and so
	\begin{equation}\label{eq42}
		u\leq v.
	\end{equation}
	
	Interchanging the roles of $u,v$ in the above argument, we also obtain
	\begin{equation}\label{eq43}
		v\leq u.
	\end{equation}
	
	From (\ref{eq42}) and (\ref{eq43}) we conclude that
	$$u=v.$$
	
	This prove the uniqueness of the solution $u\in C^1(\overline{\Omega})$ of the auxiliary problem (\ref{eq16}).
\end{proof}

Let $C^1_n(\overline{\Omega})=\{u\in C^1(\overline{\Omega}):\frac{\partial u}{\partial n}|_{\partial\Omega}=0\}$ and for every $\lambda>0$ let $K_{\lambda}:L^{\infty}(\Omega)\rightarrow C^1_n(\overline{\Omega})$
be the map which to each $\vartheta\in L^\infty(\Omega)$ assigns the unique solution $u=K_\lambda(\vartheta)\in C^1_n(\overline{\Omega})$
of the auxiliary problem (\ref{eq16}) (see Proposition \ref{prop4}). The next proposition establishes the continuity properties of this map.

\begin{prop}\label{prop5}
	If hypotheses $H(a),H(\varphi)$ hold then the map $K_{\lambda}:L^{\infty}(\Omega)\rightarrow C^1_n(\overline{\Omega})$ is sequentially continuous from $L^{\infty}(\Omega)$ furnished with the $w^*$-topology into $C^1_n(\overline{\Omega})$ with the norm topology.
\end{prop}
\begin{proof}
	Suppose that $\vartheta_n\stackrel{w^*}{\rightarrow}\vartheta$ in $L^{\infty}(\Omega)$ and let $u_n=K_{\lambda}(\vartheta_n)$, $u=K_{\lambda}(\vartheta)$.
	
	For every $n\in\NN$, we have
	\begin{eqnarray}
		&&\hat{a}(u_n)+u_n+N_{\varphi'_{\lambda}}(u_n)=\vartheta_n\label{eq44}\\
		&\Rightarrow&-{\rm div}\,(a(u_n(z))Du_n(z))+u_n(z)+\varphi'_{\lambda}(u_n(z))=\vartheta_n(z)\nonumber\\
		&&\hspace{1cm}\mbox{for almost all}\ z\in\Omega,\ \frac{\partial u_n}{\partial n}=0\ \mbox{on}\ \partial\Omega .\label{eq45}
	\end{eqnarray}
	
	On (\ref{eq44}) we act with $u_n\in C^1_n(\overline{\Omega})$. Then
	\begin{eqnarray*}
		&&\int_{\Omega}a(u_n)|Du_n|^2dz+||u_n||^2_2+\int_{\Omega}\varphi'_{\lambda}(u_n)u_ndz=\int_{\Omega}\vartheta_nu_ndz\\
		&\Rightarrow&c_1||Du_n||^2_2+||u_n||^2_2\leq c_9||u_n||\ \mbox{for some}\ c_9>0,\ \mbox{all}\ n\in\NN\\
		&&(\mbox{see hypothesis $H(a)$ and recall that}\ \varphi'_{\lambda}\ \mbox{is increasing with}\ \varphi'_{\lambda}(0)=0)\\
		&\Rightarrow&||u_n||\leq c_{10}\ \mbox{for some}\ c_{10}>0,\ \mbox{all}\ n\in\NN,\\
		&\Rightarrow&\{u_n\}_{n\geq 1}\subseteq H^1(\Omega)\ \mbox{is bounded}.
	\end{eqnarray*}
	
	By passing to a subsequence if necessary, we may assume that
	\begin{equation}\label{eq46}
		u_n\stackrel{w}{\rightarrow}\hat{u}\ \mbox{in}\ H^1(\Omega)\ \mbox{and}\ u_n\rightarrow\hat{u}\ \mbox{in}\ L^2(\Omega).
	\end{equation}
	
	Then for every $h\in H^1(\Omega)$ we have
	\begin{eqnarray}\label{eq47}
		&&\left\langle \hat{a}(u_n),h\right\rangle=\int_{\Omega}a(u_n)(Du_n,Dh)_{\RR^N}dz\rightarrow\int_{\Omega}a(\hat{u})(D\hat{u},Dh)_{\RR^N}dz=\left\langle \hat{a}(\hat{u}),h\right\rangle\nonumber\\
		&&\hspace{6cm}(\mbox{see (\ref{eq46}) and hypothesis $H(a)$}),\nonumber\\
		&\Rightarrow&\hat{a}(u_n)\stackrel{w}{\rightarrow}\hat{a}(\hat{u})\ \mbox{in}\ H^1(\Omega)^*.
	\end{eqnarray}
	
	Therefore, if in (\ref{eq44}) we pass to the limit as $n\rightarrow\infty$ and use (\ref{eq46}), (\ref{eq47}), then
	\begin{eqnarray*}
		&&\hat{a}(\hat{u})+\hat{u}+N_{\varphi'_{\lambda}}(\hat{u})=\vartheta,\\
		&\Rightarrow&\hat{u}=u\in C^1(\overline{\Omega})=\mbox{the unique solution of (\ref{eq16}) (see Proposition \ref{prop4})}.
	\end{eqnarray*}
	
	From (\ref{eq45}) and Proposition \ref{prop3}, (recall that $\{u_n\}_{n\geq 1}\subseteq H^1(\Omega)$ is bounded), we see that we can find $c_{11}>0$ such that
	\begin{equation}\label{eq48}
		||u_n||_{\infty}\leq c_{11}\ \mbox{for all}\ n\in\NN.
	\end{equation}
	
	Then (\ref{eq48}) and Theorem 2 of Lieberman \cite{13} imply that we can find $\alpha\in(0,1)$ and $c_{12}>0$ such that
	\begin{equation}\label{eq49}
		u_n\in C^{1,\alpha}(\overline{\Omega}),\ ||u_n||_{C^{1,\alpha}(\overline{\Omega})}\leq c_{12}\ \mbox{for all}\ n\in\NN.
	\end{equation}
	
	From (\ref{eq49}), the compact embedding of $C^{1,\alpha}(\overline{\Omega})$ into $C^1(\overline{\Omega})$ and (\ref{eq46}), we have
	\begin{eqnarray*}
		&&u_n\rightarrow u\ \mbox{in}\ C^1(\overline{\Omega}),\\
		&\Rightarrow&K_{\lambda}(\vartheta_n)\rightarrow K_{\lambda}(\vartheta)\ \mbox{in}\ C^1(\overline{\Omega}).
	\end{eqnarray*}
	
	This proves that $K_{\lambda}$ is sequentially continuous from $L^{\infty}(\Omega)$ with the $w^*$-topology into $C^1_n(\overline{\Omega})$ with the norm topology.
\end{proof}

We consider the following approximation to problem (\ref{eq1}):
\begin{eqnarray}\label{eq50}
	\left\{\begin{array}{ll}
		{\rm div}\,(a(u(z))Du(z))\in\varphi'_{\lambda}(u(z))+F(z,u(z),Du(z))&\mbox{in}\ \Omega,\\
		\frac{\partial u}{\partial n}=0&\mbox{on}\ \partial\Omega,\lambda>0.
	\end{array}\right\}
\end{eqnarray}
\begin{prop}\label{prop6}
	If hypotheses $H(a),H(\varphi),H(F)$ hold and $\lambda>0$, then problem (\ref{eq50}) admits a solution $u_{\lambda}\in C^1(\overline{\Omega})$.
\end{prop}
\begin{proof}
	Consider the multifunction $N:C^1_n(\overline{\Omega})\rightarrow 2^{L^{\infty}(\Omega)}$ defined by
	$$N(u)=\{f\in L^{\infty}(\Omega):f(z)\in F(z,u(z),Du(z))\ \mbox{for almost all}\ z\in\Omega\}.$$
	
	Hypotheses $H(F)(i),(ii)$ imply that the multifunction $z\mapsto F(z,u(z),Du(z))$ admits a measurable selection (see Hu and Papageorgiou \cite[p. 21]{11}) and then hypothesis $H(F)(iii)$ implies that this measurable selection belongs in $L^{\infty}(\Omega)$ and so $N(\cdot)$ has nonempty values, which is easy to see that they are $w^*$-compact (Alaoglu's theorem) and convex. Let
	$$N_1(u)=u-N(u)\ \mbox{for all}\ u\in C^1_n(\overline{\Omega}).$$
	
	We consider the following fixed point problem
	\begin{equation}\label{eq51}
		u\in K_{\lambda}N_1(u).
	\end{equation}
	
	Let $E=\{u\in C^1_n(\overline{\Omega}):u\in tK_{\lambda}N_1(u)\ \mbox{for some}\ t\in(0,1)\}$.
	\begin{claim}
		The set $E\subseteq C^1_n(\overline{\Omega})$ is bounded.
	\end{claim}
	
	Let $u\in E$. Then from the definitions of $K_{\lambda}$ and $N_1$ we have
	\begin{equation}\label{eq52}
		\hat{a}(\frac{1}{t}u)+\frac{1}{t}u+N_{\varphi'_{\lambda}}(\frac{1}{t}u)=u-f\ \mbox{with}\ f\in N(u).
	\end{equation}
	
	On (\ref{eq52}) we act with $u\in H^1(\Omega)$. Using hypothesis $H(a)$, we obtain
	\begin{eqnarray}\label{eq53}
		&&\frac{c_1}{t}||Du||^2_2+\frac{1}{t}||u||^2_2\leq||u||^2_2-\int_{\Omega}fudz\nonumber\\
		&&\hspace{1cm}(\mbox{recall that}\ \varphi'_{\lambda}\ \mbox{is increasing and}\ \varphi'_{\lambda}(0)=0),\nonumber\\
		&\Rightarrow&c_1||Du||^2_2\leq(t-1)||u||^2_2-t\int_{\Omega}fudz\leq-t\int_{\Omega}fudz\ (\mbox{recall that}\ t\in(0,1)).
	\end{eqnarray}
	
	Hypothesis $H(F)(v)$ implies that
	\begin{eqnarray}\label{eq54}
		-t\int_{\Omega}fudz\leq tc_3||u||^2_2+tc_4\int_{\Omega}|u|^2\,|Du|dz+\int_{\Omega}\gamma_3(z)|u|^2dz.
	\end{eqnarray}
	
	Let $M>0$ be as postulated by hypothesis $H(F)(iv)$. We will show that
	$$||u||_{\infty}\leq M.$$
	
	To this end let $\hat{\sigma}_0(z)=|u(z)|^2$. Let $z_0\in\overline{\Omega}$ be such that
	$$\hat{\sigma}_0(z_0)=\max\limits_{\overline{\Omega}}\hat{\sigma}_0\ (\mbox{recall that}\ u\in E\subseteq C^1_n(\overline{\Omega})).$$
	
	Suppose that $\hat{\sigma}_0(z_0)>M^2$. First assume that $z_0\in\Omega$. Then
	\begin{eqnarray*}
		&&0=D\hat{\sigma}_0(z_0)=2u(z_0)Du(z_0),\\
		&\Rightarrow&Du(z_0)=0\ (\mbox{since}\ |u(z_0)|>M).
	\end{eqnarray*}
	
	Let $\delta,\eta>0$ be as in hypothesis $H(F)(iv)$. Since $\hat{\sigma}_0(z_0)>M^2$ and $u\in C^1_n(\overline{\Omega})$ we can find $\delta_1>0$ such that
	\begin{eqnarray}\label{eq55}
		&&z\in\overline{B}_{\delta_1}(z_0)=\{z\in\Omega:|z-z_0|\leq\delta_1\}\Rightarrow|u(z)-u(z_0)|+|Du(z)|\leq\delta\nonumber\\
		&&\hspace{4cm}(\mbox{recall that}\ Du(z_0)=0),\nonumber\\
		&\Rightarrow&tf(z)u(z)+tc_1|Du(z)|^2\geq t\eta>0\ \mbox{for almost all}\ z\in\overline{B}_{\delta_1}(z_0)\\
		&&\hspace{4cm}(\mbox{see hypothesis}\ H(F)(iv)).\nonumber
	\end{eqnarray}
	
	From (\ref{eq52}) as before (see the proof of Proposition \ref{prop4}), we have
	\begin{equation}\label{eq56}
		-{\rm div}\,\left(a(\frac{1}{t}u(z))D(\frac{1}{t}u)(z)\right)+\varphi'_{\lambda}(\frac{1}{t}u(z))=(1-\frac{1}{t})u(z)-f(z)\ \mbox{for almost all}\ z\in\Omega .
	\end{equation}
	
	Using (\ref{eq56}) in (\ref{eq55}), we obtain
	\begin{eqnarray}\label{eq57}
		&&\left[{\rm div}\,\left(a(\frac{1}{t}u(z))Du(z)\right)-t\varphi'_{\lambda}(\frac{1}{t}u(z))+(t-1)u(z)\right]u(z)+tc_1|Du(z)|^2\geq t\eta\\
		&&\hspace{7cm}\mbox{for almost all}\ z\in\overline{B}_{\delta_1}(z_0).\nonumber
	\end{eqnarray}
	
	We integrate over $\overline{B}_{\delta_1}(z_0)$ and use the fact that $t\in(0,1)$. Then
	\begin{eqnarray*}
		&&\int_{\overline{B}_{\delta_1}(z_0)}{\rm div}\,(a(\frac{1}{t}u)Du)udz-t\int_{\overline{B}_{\delta_1}(z_0)}\varphi'_{\lambda}(\frac{1}{t}u)udz+tc_1\int_{\overline{B}_{\delta_1}(z_0)}|Du|^2dz\geq\mu\eta|\overline{B}_{\delta_1}(z_0)|_N\\
		&\Rightarrow&\int_{\overline{B}_{\delta_1}(z_0)}{\rm div}\,(a(\frac{1}{t}u)Du)udz+tc_1\int_{\overline{B}_{\delta_1}(z_0)}|Du|^2dz>0\\
		&&\hspace{4cm}(\mbox{recall that}\ \varphi'_{\lambda}\ \mbox{is increasing and}\ \varphi'_{\lambda}(0)=0).
	\end{eqnarray*}
	
	Using the nonlinear Green's identity (see Gasinski and Papageorgiou \cite{7}, Theorem 2.4.53, p. 210), we obtain
	$$0<-\int_{\overline{B}_{\delta_1}(z_0)}a(\frac{1}{t}u)|Du|^2dz+\int_{\partial\overline{B}_{\delta_1}(z_0)}a(\frac{1}{t}u)\frac{\partial u}{\partial n}ud\sigma+tc_1\int_{\overline{B}_{\delta_1}(z_0)}|Du|^2dz.$$
	
	Here by $\sigma(\cdot)$ we denote the $(N-1)$-dimensional Hausdorff (surface) measure defined on $\partial\Omega$. Hence we ahve
	\begin{eqnarray*}
		&&0<-c_1\int_{\overline{B}_{\delta_1}(z_0)}|Du|^2dz+\int_{\partial \overline{B}_{\delta_1}(z_0)}a(\frac{1}{t}u)\frac{\partial u}{\partial n}ud\sigma+tc_1\int_{\overline{B}_{\delta_1}(z_0)}|Du|^2dz\\
		&&\hspace{1cm}(\mbox{see hypothesis $H(a)$}),\\
		&\Rightarrow&0<\int_{\partial \overline{B}_{\delta_1}(z_0)}a(\frac{1}{t}u)\frac{\partial u}{\partial n}ud\sigma\ (\mbox{recall that}\ t\in(0,1)),\\
		&\Rightarrow&0<c_2\int_{\partial \overline{B}_{\delta_1}(z_0)}\frac{\partial u}{\partial n}ud\sigma\ (\mbox{see hypothesis $H(a)$}).
	\end{eqnarray*}
	
	Thus we can find a continuous path $\{c(t)\}_{t\in[0,1]}$ in $\overline{B}_{\delta_1}(z_0)$ with $c(0)=z_0$ such that
	\begin{eqnarray*}
		&&a<\int^1_0u(c(t))(Du(c(t)),c'(t))_{\RR^N}dt\\
		&&=\int^1_0\frac{1}{2}\frac{d}{dt}u(c(t))^2dt\\
		&&=\frac{1}{2}[u(c(1))-u(z_0)],\\
		&\Rightarrow&u(z_0)<u(c(1)),
	\end{eqnarray*}
	which contradicts the choice of $z_0$. So, we cannot have $z_0\in\Omega$.
	
	Therefore we assume that $z_0\in\partial\Omega$. Since $u\in C^1_n(\overline{\Omega})$, again we have $Du(z_0)=0$ and so the above argument applies with $\partial \overline{B}_{\delta_1}(z_0)$ replaced by $\partial \overline{B}_{\delta_1}(z_0)\cap\Omega$.
	
	Hence we have proved that
	\begin{eqnarray}\label{eq58}
		||u||_{\infty}\leq M\ \mbox{for all}\ u\in E\ (\mbox{here}\ M>0\ \mbox{is as in hypothesis}\ H(F)(iv)).
	\end{eqnarray}
	
	We use (\ref{eq58}) in (\ref{eq54}) and have
	\begin{eqnarray}\label{eq59}
		&&-t\int_{\Omega}fudz\leq tc_{13}(1+||Du||_2)\ \mbox{for some}\ c_{13}>0,\nonumber\\
		&\Rightarrow&c_1||Du||^2_2\leq c_{13}(1+||Du||_2)\ (\mbox{see (\ref{eq53}) and recall}\ t\in(0,1)),\nonumber\\
		&\Rightarrow&||Du||_2\leq c_{14}\ \mbox{for some}\ c_{14}>0,\ \mbox{all}\ u\in E.
	\end{eqnarray}
	
	Then (\ref{eq58}), (\ref{eq59}) imply that $E\subseteq H^1(\Omega)$ is bounded. Invoking Theorem 2 of Lieberman \cite{13}, we can find $c_{15}>0$ such that
	\begin{eqnarray*}
		&&||u||_{C^1(\overline{\Omega})}\leq c_{15}\ \mbox{for all}\ u\in E,\\
		&\Rightarrow&E\subseteq C^1_n(\overline{\Omega})\ \mbox{is bounded}.
	\end{eqnarray*}
	
	This proves Claim 2.
	
	Recall that hypotheses $H(F)(i),(ii),(iii)$ imply that $N_1$ is a multifunction which is usc from $C^1_n(\overline{\Omega})$ with the norm topology into $L^{\infty}(\Omega)$ with the $w^*$-topology (see Hu and Papageorgiou \cite[p. 21]{11}). This fact, Proposition \ref{prop5} and Claim 2, permit the use of Theorem \ref{th1}. So, we can find $u_{\lambda}\in C^1_n(\overline{\Omega})$ such that
	\begin{eqnarray*}
		&&u_{\lambda}\in K_{\lambda}N_1(u_{\lambda}),\\
		&\Rightarrow&u_{\lambda}\in C^1_n(\overline{\Omega})\ \mbox{is a solution of problem (\ref{eq50})}.
	\end{eqnarray*}
\end{proof}

Now we are ready for the existence theorem concerning problem (\ref{eq1}).
\begin{theorem}\label{th7}
	If hypotheses $H(a),H(\varphi),H(F)$ hold, then problem (\ref{eq1}) admits a solution $u\in C^1_n(\overline{\Omega})$.
\end{theorem}
\begin{proof}
	Let $\lambda_n\rightarrow 0^+$. From Proposition \ref{prop6}, we know that problem (\ref{eq50}) (with $\lambda=\lambda_n$) has a solution $u_n=u_{\lambda_n}\in C^1_n(\overline{\Omega})$. Moreover, from the proof of that proposition, we have
	\begin{equation}\label{eq60}
		||u_n||_{\infty}\leq M\ \mbox{for all}\ n\in\NN\ (\mbox{see (\ref{eq58})}).
	\end{equation}
	
	For every $n\in\NN$, we have
	\begin{equation}\label{eq61}
		\hat{a}(u_n)+N_{\varphi'_{\lambda_n}}(u_n)+f_n=0\ \mbox{with}\ f_n\in N(u_n)\ (\mbox{see the proof of Proposition \ref{prop6}}).
	\end{equation}
	
	On (\ref{eq61}) we act with $u_n$ and obtain
	\begin{eqnarray}\label{eq62}
		&&c_1||Du_n||^2_2\leq||f_n||_2||u_n||_2\nonumber\\
		&&(\mbox{see hypothesis $H(a)$ and recall that}\ \varphi'_{\lambda}(s)s\geq 0\ \mbox{for all}\ s\in\RR),\nonumber\\
		&\Rightarrow&||Du_n||_2\leq c_{16}\ \mbox{for some}\ c_{16}>0,\ \mbox{all}\ n\in\NN\\
		&&(\mbox{see (\ref{eq60}) and hypothesis}\ H(F)(iii)).\nonumber
	\end{eqnarray}
	
	From (\ref{eq60}) and (\ref{eq62}) it follows that
	$$\{u_n\}_{n\geq 1}\subseteq H^1(\Omega)\ \mbox{is bounded}.$$
	
	So, by passing to a subsequence if necessary, we may assume that
	$$u_n\stackrel{w}{\rightarrow}u\ \mbox{in}\ H^1(\Omega)\ \mbox{and}\ u_n\rightarrow u\ \mbox{in}\ L^2(\Omega).$$
	
	Acting on (\ref{eq61}) with $N_{\varphi'_{\lambda_n}}(u_n)(\cdot)=\varphi'_{\lambda_n}(u_n(\cdot))\in C(\overline{\Omega})\cap H^1(\Omega)$ (recall that $\varphi'_{\lambda_n}(\cdot)$ is Lipschitz continuous and see Marcus and Mizel \cite{15}), we have
	\begin{equation}\label{eq63}
		\int_{\Omega}a(u_n)(Du_n,D\varphi'_{\lambda}(u_n))_{\RR^N}dz+
||N_{\varphi'_{\lambda_n}}(u_n)||^2_2=-\int_{\Omega}f_n\varphi'_{\lambda_n}(u_n)dz.
	\end{equation}
	
	From the chain rule of Marcus and Mizel \cite{15}, we have
	\begin{equation}\label{eq64}
		D\varphi'_{\lambda_n}(u_n)=\varphi''_{\lambda_n}(u_n)Du_n.
	\end{equation}
	
	Since $\varphi'_{\lambda_n}(\cdot)$ is increasing (recall that $\varphi_{\lambda_n}$ is convex), we have
	\begin{equation}\label{eq65}
		\varphi''_{\lambda_n}(u_n(z))\geq 0\ \mbox{for almost all}\ z\in\Omega .
	\end{equation}
	
	Using (\ref{eq64}), (\ref{eq65}) and hypothesis $H(a)$, we see that
	\begin{equation}\label{eq66}
		0\leq\int_{\Omega}a(u_n)(Du_n,D\varphi'_{\lambda_n}(u_n))_{\RR^N}dz.
	\end{equation}
	
	Using (\ref{eq66}) in (\ref{eq63}), we obtain
	\begin{eqnarray*}
		&&||N_{\varphi'_{\lambda_n}}(u_n)||^2_2\leq||f_n||_2||N_{\varphi'_{\lambda_n}}(u_n)||_2\ \mbox{for all}\ n\in\NN,\\
		&\Rightarrow&||N_{\varphi'_{\lambda_n}}(u_n)||_2\leq||f_n||_2\leq c_{17}\ \mbox{for some}\ c_{17}>0,\ \mbox{all}\ n\in\NN\\
		&&(\mbox{see (\ref{eq60}) and hypothesis H(F)(iii)})\\
		&\Rightarrow&\{N_{\varphi'_{\lambda_n}}(u_n)\}_{n\geq 1}\subseteq L^2(\Omega)\ \mbox{is bounded}.
	\end{eqnarray*}
	
	So, we may assume that
	\begin{equation}\label{eq68}
		N_{\varphi'_{\lambda_n}}\stackrel{w}{\rightarrow}g\ \mbox{and}\ f_n\stackrel{w}{\rightarrow}f\ \mbox{in}\ L^2(\Omega).
	\end{equation}
	
	As in the proof of Proposition \ref{prop5} (see (\ref{eq47})), we show that
	\begin{equation}\label{eq69}
		\hat{a}(u_n)\stackrel{w}{\rightarrow}\hat{a}(u)\ \mbox{in}\ H^1(\Omega)^*.
	\end{equation}
	
	So, if in (\ref{eq61}) we pass to the limit as $n\rightarrow\infty$ and use (\ref{eq68}) and (\ref{eq69}), we obtain
	\begin{eqnarray}\label{eq70}
		&&\hat{a}(u)+g+f=0,\nonumber\\
		&\Rightarrow&-{\rm div}\,(a(u(z))Du(z))+g(z)+f(z)=0\ \mbox{for almost all}\ z\in\Omega,\ \frac{\partial u}{\partial n}=0\ \mbox{on}\ \partial\Omega\\
		&&(\mbox{see the proof of Proposition \ref{prop4}}).\nonumber
	\end{eqnarray}
	
	Because of (\ref{eq60}) and Theorem 2 of Lieberman \cite{13}, we know that there exist $\alpha\in(0,1)$ and $c_{18}>0$ such that
	\begin{eqnarray}\label{eq71}
		&&u_n\in C^{1,\alpha}(\overline{\Omega}),\ ||u_n||_{C^{1,\alpha}(\overline{\Omega})}\leq c_{18}\ \mbox{for all}\ n\in\NN,\nonumber\\
		&\Rightarrow&u_n\rightarrow u\ \mbox{in}\ C^1(\overline{\Omega})\ (\mbox{recall that}\ C^{1,\alpha}(\overline{\Omega})\ \mbox{is embedded compactly into}\ C^1(\overline{\Omega})).
	\end{eqnarray}
	
	Recall that
	\begin{eqnarray}\label{eq72}
		&&f_n(z)\in F(z,u_n(z),Du_n(z))\ \mbox{for almost all}\ z\in\Omega,\ \mbox{all}\ n\in\NN,\nonumber\\
		&\Rightarrow&f(z)\in F(z,u(z),Du(z))\nonumber\\
		&&(\mbox{see (\ref{eq68}), (\ref{eq71}), hypothesis H(F)(ii) and Proposition 6.6.33, p. 521 of \cite{20}}),\nonumber\\
		&\Rightarrow&f\in N(u).
	\end{eqnarray}
	
	Also, from (\ref{eq68}), (\ref{eq71}) and Corollary 3.2.51, p. 179 of \cite{20}, we have
	\begin{equation}\label{eq73}
		g(z)\in\partial\varphi(u(z))\ \mbox{for almost all}\ z\in\Omega.
	\end{equation}
	
	So, from (\ref{eq70}), (\ref{eq72}), (\ref{eq73}) we conclude that $u\in C^1_n(\overline{\Omega})$ is a solution of problem (\ref{eq1}).
\end{proof}

\section{Examples}

In this section we present two concrete situations illustrating our result.

For the first, let $\mu\leq 0$ and consider the function
$$\varphi(x)=\left\{\begin{array}{ll}
	+\infty&\mbox{if}\ x<\mu\\
	0&\mbox{if}\ \mu\leq x.
\end{array}\right.$$

Evidently we have
$$\varphi\in\Gamma_0(\RR)\ \mbox{and}\ 0\in\partial\varphi(0).$$

In fact note that
$$\partial \varphi(x)=\left\{\begin{array}{ll}
	\emptyset&\mbox{if}\ x<\mu\\
	\RR_-&\mbox{if}\ x=\mu\\
	\{0\}&\mbox{if}\ \mu<x.
\end{array}\right.$$

Also consider a Carath\'eodory function $f:\Omega\times\RR\times\RR^N\rightarrow\RR$ which satisfies hypotheses $H(F)(iii),(iv),(v)$. For example, we can have the following function (for the sake of simplicity we drop the $z$-dependence):
$$f(x,\xi)=c\sin x+x-\ln(1+|\xi|)+\vartheta\ \mbox{with}\ c_1\vartheta>0.$$

Then according to Theorem \ref{th7}, we can find a solution $u_0\in C^1(\overline{\Omega})$ for the following problem:
$$\left\{\begin{array}{l}
	{\rm div}\,(a(u(z))Du(z))\leq f(z,u(z),Du(z))\ \mbox{for almost all}\  z\in\{u=\mu\},\\
	{\rm div}\,(a(u(z))Du(z))=f(z,u(z),Du(z))\ \mbox{for almost all}\  z\in\{\mu<u\},\\
	u(z)\geq\mu\ \mbox{for all}\ z\in\overline{\Omega},\frac{\partial u}{\partial n}=0\ \mbox{on}\ \partial\Omega .
\end{array}\right\}$$

For the second example, we consider a variational-hemivariational inequality. Such problems arise in mechanics, see Panagiotopoulos \cite{19}. So, let $j(z,x)$ be a locally Lipschitz integrand (that is, for all $x\in\RR$, $z\mapsto j(z,x)$ is measurable and for almost all $z\in\Omega$, $x\mapsto j(z,x)$ is locally Lipschitz). By $\partial_cj(z,x)$ we denote the Clarke subdifferential of $j(z,\cdot)$. We impose the following conditions on the integrand $j(z,x)$:
\begin{itemize}
	\item[(a)] for almost all $z\in\Omega$, all $x\in\RR$ and all $v\in\partial j(z,x)$
	$$|v|\leq\hat{c}_1(1+|x|)\ \mbox{for almost all}\ z\in\Omega,\ \mbox{all}\ x\in\RR,\ \mbox{with}\ \hat{c}_1>0;$$
	\item[(b)] $0<\hat{c}_2\leq\liminf\limits_{x\rightarrow\pm\infty}\frac{v}{x}\leq\limsup\limits_{x\rightarrow\pm\infty}\frac{v}{x}\leq\hat{c}_3$ uniformly for almost all $z\in\Omega$, all $v\in\partial j(z,x)$
	\item[(c)] $-\hat{c}_4\leq\liminf\limits_{x\rightarrow 0}\frac{v}{x}\leq\limsup\limits_{x\rightarrow 0}\frac{v}{x}\leq\hat{c}_5$ uniformly for almost all $z\in\Omega$, all $v\in\partial j(z,x)$ and with $\hat{c}_4,\hat{c}_5>0$.
\end{itemize}

A possible choice of $j$ is the following (as before for the sake of simplicity we drop the $z$-dependence):
$$j(x)=\left\{\begin{array}{ll}
	\frac{1}{p}|x|^p-\cos(\frac{\pi}{2}|x|)&\mbox{if}\ |x|\leq 1\\
	\frac{1}{2}x^2-\ln|x|+c&\mbox{if}\ 1<|x|
\end{array}\right.\ \mbox{with}\ c=\frac{1}{p}-\frac{1}{2},\ 1<p.$$

We set
$$F(z,x,\xi)=\partial j(z,x)+x|\xi|+\vartheta(z)\ \mbox{with}\ \vartheta\in L^{\infty}(\Omega).$$
 	
	Using (a),(b),(c) above, we can see that hypotheses $H(F)$ are satisfied.
	
	Also, suppose that $\varphi$ satisfies hypothesis $H(\varphi)$. Two specific choices of interest are
	$$\varphi(x)=|x|\ \mbox{and}\ \varphi(x)=i_{[-1,1]}(x)=\left\{\begin{array}{ll}
		0&\mbox{if}\ |x|\leq 1\\
		+\infty&\mbox{if}\ 1<|x|.
	\end{array}\right.$$
	
	Then the following problem admits a solution $u_0\in C^1(\overline{\Omega})$:
	$$\left\{\begin{array}{ll}
		{\rm div}\,(a(u(z))Du(z))\in\partial\varphi(u(z))+F(z,u(z),Du(z))&\mbox{in}\ \Omega,\\
		\frac{\partial u}{\partial n}=0&\mbox{on}\ \partial\Omega.
	\end{array}\right\}$$
	
	The case $\varphi\equiv 0$ (hemivariational inequalities) incorporates problems with discontinuities in which we fill-in the gaps at the jump discontinuities.
	
\medskip
{\bf Acknowledgments.} V. R\u{a}dulescu was supported by a grant of the Romanian National
Authority for Scientific Research and Innovation, CNCS-UEFISCDI, project number PN-II-PT-PCCA-2013-4-0614. D. Repov\v{s}
was supported by the Slovenian Research Agency grants P1-0292-0101, J1-6721-0101, J1-7025-0101 and J1-5435-29-0101.

\end{document}